\documentclass[11pt,twoside]{amsart}

\usepackage[english]{babel}
\usepackage{amsmath, amsthm, amssymb, amsfonts}
\usepackage{latexsym}
\usepackage{graphicx}
\usepackage{pdfpages}
\usepackage {bm}
\usepackage {indentfirst} 

\usepackage{hyperref,cite}

\advance\hoffset by -1.5cm

\newcommand{\D}{\mathcal{D}}

\newcommand{\B}{\mathcal{B(H)}}
\newcommand{\R}{\mathcal{R}}
\newcommand{\n}{\mathcal{N}}
\newcommand{\h}{\mathcal{H}}
\newcommand{\m}{\mathcal{M}}
\newcommand{\ka}{\mathcal{K}}

\newtheorem{example}{Example}[section]
\newtheorem{theorem}[example]{Theorem}

\newtheorem{proposition}[example]{Proposition}
\newtheorem{corollary}[example]{Corollary}
\newtheorem{remark}[example]{Remark}

\numberwithin{equation}{section}

\begin{document}

\title[The Cauchy dual and $2$-isometric liftings]{The Cauchy dual and $2$-isometric liftings of\\ concave operators}

 \author[C. Badea]{Catalin Badea}
  \address{Univ. Lille, CNRS, UMR 8524 - Laboratoire Paul Painlev\'{e}, Lille, France}
\email{catalin.badea@univ-lille.fr}
\urladdr{http://math.univ-lille1.fr/~badea/}

\author[L. Suciu]{Laurian Suciu}
 \address{Department of Mathematics and Informatics, ``Lucian Blaga'' University of Sibiu, Dr. Ion Ra\c{t}iu 5-7, Sibiu, 550012, Romania}
 \email{laurians2002@yahoo.com}

\begin{abstract}
We present some $2$-isometric lifting and extension results for Hilbert space concave operators. For a special class of concave operators we study their Cauchy dual operators and discuss conditions under which these operators are subnormal. In particular,  the quasinormality of compressions of such operators is studied.
\end{abstract}

\maketitle

\section{Introduction and preliminaries}

\subsection*{Preamble} Extensions and liftings are classical notions in Operator Theory.
To give some examples, we recall that a linear bounded Hilbert space operator is an isometry if and only if it is the restriction of a unitary operator to an invariant subspace. Also, it is known from the Sz.-Nagy-Foias dilation theory that an operator $C$ is a contraction if and only if it lifts to an isometry $V$; that is if and only if its adjoint $C^*$ is the restriction of a coisometry $V^*$ to an invariant subspace (see \cite{FF,SFb}).

In this paper, we prove some $2$-isometric lifting and extension results for Hilbert space concave operators, that is for operators satisfying the inequality \eqref{eq11} below. A $2$-isometry is an operator for which the equality in \eqref{eq11} holds true.

The notion of Cauchy dual operator for a left invertible operator is more recent, being introduced in $2001$ by Shimorin in his seminal study \cite{Sh} of Wold-type decompositions and wandering subspaces. Here we study the Cauchy dual operators for the special class of concave operators satisfing the condition \eqref{eq13} below.

\subsection*{Notation and basic definitions} For a complex Hilbert spaces $\h$ we denote by $\B$ the Banach algebra of all bounded linear operators on $\h$ with the unit element $I=I_{\h}$ (the identity operator). For $T\in \B$ the kernel and the (closed) range of $T$ are denoted by $\n(T)$ respectively $\overline{\R(T)}$. Also, $T^*\in \B$ stands for the adjoint operator of $T$, and the orthogonal projection in $\B$ onto a closed subspace $\m \subset \h$ is denoted by $P_{\m}$. For $T\in \B$ we consider the operator $\Delta_T:=T^*T-I$. The operator $T$ is called {\it expansive} (respectively {\it contractive}) if $\Delta_T\ge 0$ (respectively $\Delta_T\le 0$). If $T$ is a contraction, then $D_T=-\Delta_T$ is the defect operator and $\mathcal{D}_T=\overline{\R(D_T)}$ is the defect space of $T$.

Recall that an operator $T\in \B$ is said to be \emph{normal} if $TT^*=T^*T$, \emph{quasinormal} if $TT^*T=T^*T^2$, \emph{hyponormal} if $TT^*\le T^*T$, an \emph{isometry} if $T^*T=I$ and, finally, $T$ is \emph{unitary} if it is a normal isometry.

A (closed) subspace $\h_0\subset \h$ is \emph{invariant} for $T\in \B$ if $T\h_0\subset \h_0$, and $\h_0$ is \emph{reducing} for $T$ if $T\h_0 \subset \h_0$ and $T^*\h_0 \subset \h_0$. If $T\in \B$ and $\h$ is a closed subspace of another Hilbert space $\ka$, then $S\in \mathcal{B}(\ka)$ is an {\it extension} of $T$ if $\h$ is invariant for $S$ and $T=S|_{\h}$. This definition can be rephrased as $J_{\h,\ka}T=SJ_{\h,\ka}$ where $J_{\h,\ka}:\h\to \ka$ is the natural embedding of $\h$ into $\ka$. We also say that $S$ is a {\it lifting} of $T$ if $S^*$ is an extension of $T^*$; that is if $P_{\ka,\h}S=TP_{\ka,\h}$ where $P_{\ka,\h}=J^*_{\h,\ka}$ is the projection of $\ka$ onto $\h$.

An operator $T$ on $\h$ is called \emph{subnormal} if it has a normal extension on a Hilbert space $\ka\supset \h$.

Recall (\cite{Sh}, \cite{Ch, CC}, \cite{KFAH}) that an operator $T$ on $\h$ is called {\it concave} if it satisfies the inequality
\begin{equation}\label{eq11}
T^{*2}T^2-2T^*T+I\le 0.
\end{equation}
The operator $T$ is said to be a 2-{\it isometry} whenever the equality in \eqref{eq11} holds true. In this case, according to \cite{Ag, AS1, AS2}, the above operator $\Delta_T = T^*T-I$ is called the {\it covariance operator} of $T$, while the scalar ${\rm cov}(T):=\|\Delta_T\|^{1/2}$ is called the {\it covariance} of $T$. It is obvious from the inequality \eqref{eq11} that $\Delta_T\ge 0$, i.e. $T$ is expansive, hence $T$ is left invertible.

For a positive operator $A\in \B$ and an integer $m\ge 1$ we define the operator
\begin{equation}\label{eq12}
B_A^m(T):=\sum_{j=0}^m (-1)^j\begin{pmatrix} m\\ j \end{pmatrix} T^{*j}AT^j, \quad T\in \B.
\end{equation}

Accordingly to \cite{JKKL} we say that $T$ is $(A,m)$-{\it expansive} (or $(A,m)$-{\it contractive}) if $B_A^m(T)\le 0$ ($B_A^m(T)\ge 0$). Also, $T$ is said to be $(A,m)$-{\it hyperexpansive} (or $(A,m)$-{\it hypercontractive}) if $B_A^n(T)\le 0$ ($B_A^n(T)\ge 0$) for all positive integers $n\le m$. Finally, $T$ is called {\it completely} $A$-{\it hyperexpansive} ({\it completely} $A$-{\it hypercontractive}) if $B_A^m(T)\le 0$ ($B_A^m(T)\ge 0$) for all $m\ge 1$. When $A=I$ we denote $B_m(T)=B_I^m(T)$ and we shortly refer to m-(hyper)expansivity instead of $(I,m)$-(hyper)expansivity etc.

In particular, following the terminology of \cite{CFM, MMS, S-2006, S-2009} we say that $T$ is an $A$-{\it contraction} if $T$ is $(A,1)$-contractive,\, i.e. $T^*AT\le A$, and $T$ is an $A$-{\it isometry} when $T^*AT=A$. Every $A$-isometry is completely $A$-hypercontractive (and completely $A$-hyperexpansive). In this context, the inequality \eqref{eq11} can be written as $T^*\Delta_T T\le \Delta_T$. Hence, concave operators are $\Delta_T$-contractions, or $2$-hyperexpansive operators, while  $2$-isometries are $\Delta_T$-isometries.

Such general classes of operators were studied by many authors, from several points of view. We refer the reader to \cite{Ag}, \cite{AS1, AS2}, \cite{Al}, \cite{ACJS}, \cite{At}, \cite{Ch, CC}, \cite{CFM}, \cite{Ex}, \cite{JKKL}, \cite{KFAH}, \cite{MMS, MS}, \cite{McC}, \cite{Ol}, \cite{R}, \cite{SS}, \cite{Sh}, \cite{S-2006, S-2009} for some of these contributions.

\subsection*{Organization of the paper} In Section 2 we construct several $2$-isometric liftings for a given concave operator. Also, we show that concave operators have certain extensions with block matrices containing contractions and unitary operators on the main diagonal, so having their spectrum in the closed unit disc (as in \cite {McC}, \cite{Ba}). Finally, we characterize concave operators $T$ which are $\Delta_T$-{\it regular}, that is they satisfy the condition
\begin{equation}\label{eq13}
\Delta_TT=\Delta_T^{1/2}T\Delta_T^{1/2}.
\end{equation}

In Section 3 we study properties of the Cauchy dual $T'=T(T^*T)^{-1}$ associated to a given concave operator $T$. The Cauchy dual operator was studied in \cite{Sh}, \cite{Ch}, \cite{CC} and recently in \cite{ACJS}. We describe the $\Delta_T$-regular concave operators in terms of $T'$ and we solve the Cauchy dual subnormality problem (which appear in \cite{ACJS} for $2$-isometries) for this class of operators. Recall that J. Agler showed in \cite{Ag} that a contraction $T$ is subnormal if and only if it is completely hypercontractive, while in \cite{ACJS} it was proved that if $T$ is a $\Delta_T$-regular $2$-isometry, then $T'$ is subnormal. Here we show that for $T$ concave and $\Delta_T$-regular the Cauchy dual contraction $T'$ is subnormal if and only if $T$ is completely hyperexpansive.

In Section 4 we present several conditions which are equivalent to the quasinormality of the compression to $\overline{\R(\Delta_T)}$ of a $\Delta_T$-regular concave operator $T$. The same thing is done for the compression of $T'$ to $\overline{\R(\Delta_T)}$. As usual, by the compression of $T$ to a subspace $\h_0\subset \h$ we mean the operator $P_{\h_0}T|_{\h_0}$.

\medskip

\section{Liftings and extensions of concave operators}
\medskip

We proceed to the construction of two special 2-isometric liftings for the class of concave operators. Recall that a lifting $S$ on $\ka \supset \h$ is said to be \emph{minimal} if $\ka=\bigvee_{n\ge 0} S^n\h.$

\begin{theorem}\label{te21}
Let $T\in \B$ be a concave operator. Then:
\begin{itemize}
\item[(i)] $T$ has a 2-isometric lifting $S$ on a Hilbert space $\ka=\h^{\perp} \oplus \h$ with the covariance ${\rm cov}(S)=\sqrt{2}{\rm max}\{1,\|\Delta_T\|^{1/2}\}$ and having a block matrix of the form
    \begin{equation}\label{eq21}
    S=
    \begin{pmatrix}
    W & X\\
    0 & T
    \end{pmatrix}
    \begin{bmatrix}
    \h^{\perp}\\
    \h
    \end{bmatrix},
    \quad W=
    \begin{pmatrix}
    S_+^1 & \sqrt{2}E_0\\
    0 & S_+^0
    \end{pmatrix}
    \begin{bmatrix}
    \h_1\\
    \h_0
    \end{bmatrix},
    \quad W^*X=0,
    \end{equation}
where $W$ is a 2-isometry on a nontrivial decomposition $\h^{\perp} = \h_1\oplus \h_0$ with $S_+^0,S_+^1$ unilateral shifts and $E_0$ is an isometry.

\item[(ii)] $T$ has a minimal 2-isometric lifting $S$, with $W=S|_{\h^{\perp}}$ in \eqref{eq21} an isometry such that $W^*X=0$ and ${\rm cov}(S)=\|\Delta_T\|^{1/2}$.
\end{itemize}
\end{theorem}

\begin{proof}
Assume $T$ concave, that is $T^*\Delta_TT\le \Delta_T$. Let $\Omega_T:=\Delta_T-T^*\Delta_TT$. Clearly, we may assume $\Delta_T\neq 0$ and $\Omega_T\neq 0$, otherwise $T$ is an isometry or a 2-isometry, respectively. Then it is easy to see that for every $h\in \h$ and any integer $n\ge 1$ one has
$$
\|\Delta_T^{1/2}h\|^2=\sum_{j=0}^n\|\Omega_T^{1/2}T^jh\|^2+\|\Delta_T^{1/2}T^{n+1}h\|^2.
$$
Since $T^{*(n+1)}\Delta_TT^{n+1}\le T^{*n}\Delta_TT^n\le \Delta_T$, the sequence $\{T^{*n}\Delta_TT^n\}$ converges strongly to an operator $A_T\ge 0$ such that $T^*A_TT=A_T$. So, from the above equality we obtain the relation
$$
\|\Delta_T^{1/2}h\|^2=\sum_{j=0}^{\infty} \|\Omega_T^{1/2}T^jh\|^2+\|A_Th\|^2,
$$
whence we have $\overline{\Delta_T\h}\subset \h_0 \oplus \overline{A_T\h}$ where $\h_0=l_+^2(\overline{\Omega_T\h})$.

For the construction of $S$ in (i) we use the operator $\Delta_T$. The previous equality suggests that one can firstly obtain a one step lifting $T_0$ of $T$ on $\h\oplus \h_0$ and later a second step lifting of $T$ (one step for $T_0$) on $\h\oplus \h_0 \oplus \h_1$ where $\h_1=l_+^2(\h_0\oplus \overline{(\Delta_T-\Omega_T)\h})$. Clearly, since $\Delta_T-\Omega_T=T^*\Delta_TT$, we have $(\Delta_T-\Omega_T)\h \subset \h_0\oplus A_T\h$.

We define $S$ on $\ka=\h_1\oplus (\h_0\oplus \h)=\h_1\oplus \h_0\oplus \h$ by the block matrices
\begin{equation}\label{eq22}
S=
\begin{pmatrix}
S_+^1& E\\
0 & T_0
\end{pmatrix}
=
\begin{pmatrix}
S_+^1 & \sqrt{2}E_0 & J_1(\Delta_T-\Omega_T)^{1/2}\\
0 & S_+^0 & J_0\Omega_T^{1/2}\\
0 & 0 & T
\end{pmatrix}.
\end{equation}
Here $S_+^0, S_+^1$ are the forward shifts on $\h_0$, respectively $\h_1$, while $E_0,J_1,J_0$ are the embedding mappings of $\h_0$ and
$\overline{(\Delta_T-\Omega_T)\h}$ into $\h_1$, respectively of $\overline{\Omega_T \h}$ into $\h_0$. Also, the operator $E: \h_0 \oplus \h\to \h_1$ and the lifting $T_0$ of $T$ on $\h_0\oplus \h$ have the matrix representations which appear in the second matrix of $S$ in \eqref{eq22}, respectively.

Since $S_+^{1*}E=0$ and $S_+^{0*}J_0\Omega_T^{1/2}=0$, we have
$$
\Delta_{T_0}=0\oplus (2\Delta_T-T^*\Delta_TT), \quad \Delta_S=0\oplus [(2I_{\h_0}\oplus T^*\Delta_TT)+\Delta_{T_0}]=0\oplus 2(I_{\h_0}\oplus \Delta_T).
$$

Now a simple computation gives the equality $S^*\Delta_SS=\Delta_S$, i.e. $S$ is a 2-isometry and by \eqref{eq22}, $S$ is a lifting of $T$. Obviously, $S$ can be expressed in the terms of $W$ and $X$ as in \eqref{eq21} with $W^*X=0$, $W$ being a 2-isometry (as the restriction of $S$ to its invariant subspace $\h_1 \oplus \h_0$). Also, from the above representation of $\Delta_S$ we get ${\rm cov}(S)=\|\Delta_S\|^{1/2}=\sqrt{2} {\rm max}\{1, \|\Delta_T\|^{1/2}\}$ taking into account that $\Delta_T\ge 0$. The assertion (i) is now proved.

To show the assertion (ii) we use that $T$ is expansive, i.e. $T^*T\ge I$. Thus, by a result of Treil and Volberg (see \cite{BFF}, \cite{TV}), there exist a Hilbert space $\ka'\supset \h$, an isometry $V'$ on $\ka'$ and an operator $B:\h \to \ka'$ such that $BT=V'B$, $P_{\h}B=\Delta_T^{1/2}$ and $\|B\|=\|\Delta_T\|^{1/2}$. Then the operator $A=B^*B$ satisfies the conditions: $T^*AT=A$, $\Delta_T \le A$ and $\|A\|=\|\Delta_T\|$. Now we define the lifting $S_0$ of $T$ on the space $\ka_0=\h^{\perp}\oplus \h$, where $\h^{\perp}=l_+^2(\overline{(A-\Delta_T)\h})$, by
\begin{equation}\label{eq23}
S_0=
\begin{pmatrix}
S_+ & J(A-\Delta_T)^{1/2}\\
0 & T
\end{pmatrix}.
\end{equation}
Here $S_+$ is the forward shift on $\h^{\perp}$ and $J$ is the embedding mapping of $\overline{(A-\Delta_T)\h}$ into $\h^{\perp}$. It is clear that $\Delta_{S_0}= 0\oplus A$, which immediately gives $S_0^*\Delta_{S_0}S_0=\Delta_{S_0}$. Thus $S_0$ is a $2$-isometry. Also, one has ${\rm cov}(S_0)=\|A\|^{1/2}=\|\Delta_T\|^{1/2}$. By a standard argument one can see that the $2$-isometric lifting $S_0$ for $T$ is minimal, that is it satisfies the condition $\ka_0=\bigvee_{n\ge 0} S_0^n\h$. This ends the proof.
\end{proof}

The converses of the statements (i) and (ii) in Theorem \ref{te21} are not true. In other words, the class of operators which have 2-isometric liftings as in (i) and (ii) above is larger than that of concave operators. For example, let $T$ be a $T^*T$-isometry, i.e. $T^{*2}T^2=T^*T$. Let $S$ be the operator on $\widehat{\h}=l_+^2(\h) \oplus \h$ with the matrix representation
$$
S=
\begin{pmatrix}
S_+ & J\\
0 & T
\end{pmatrix},
$$
where $S_+$ is the forward shift on $l_+^2(\h)$ and $J$ is the canonical embedding of $\h$ into $l_+^2(\h)$. It is clear that $S$ is a minimal 2-isometric lifting for $T$ as in Theorem \ref{te21} (ii). But $T$ is not necessary concave. In fact, since $T|_{\overline{\R(T)}}$ is an isometry, one can easily see that $T$ is concave if and only if the operator $T_1=P_{\overline{\R(T)}}T|_{\n(T^*)}$ is expansive.

We mention without further details that if $T$ satisfies the inequality $T^{*2}T^2\le T^*T$, then $T$ has a 2-isometric lifting as in Theorem \ref{te21}, (i).

An interesting problem in this context is to describe the class of all operators having 2-isometric liftings; presently we do not know the answer.

Returning to Theorem \ref{te21}, we remark that the lifting in (i) is not minimal, in general. In this case the lifts produced by (i) and (ii) are not unitarily equivalent.

The minimal lifting from (ii) leads to an extension of $T$ with interesting properties.

\begin{proposition}\label{pr22}
Every concave operator $T$ on $\h$ has an extension $\widetilde{T}$ on a Hilbert space $\m\supset \h$, which with respect to a decomposition $\m=\m_0\oplus \m_1$ has the form
\begin{equation}\label{eq24}
\widetilde{T}=
\begin{pmatrix}
C & \delta E\\
0 & U
\end{pmatrix},
\end{equation}
where $C,E$ are contractions, $U$ is unitary and $\delta=\|\Delta_T\|^{1/2}$, such that there exist a Hilbert space $\m'$ and two isometries $J_C:\D_C\to \m'$, $J_E:\D_E\to \m'$ satisfying the condition
\begin{equation}\label{eq25}
C^*E+D_CJ_C^*J_ED_E=0.
\end{equation}

\end{proposition}

\begin{proof}
Clearly, one can assume $T$ non-isometric, i.e. $\delta=\|\Delta_T\|^{1/2}>0$. Let $S_0$ be the 2-isometric lifting of $T$ with ${\rm cov(S_0)}=\delta$ given by \eqref{eq23} on $\ka_0=\h^{\perp}\oplus \h$, and let $\widetilde{S}$ be a Brownian unitary extension of the 2-isometry $S_0$ on $\ka=\ka_0\oplus \ka_1$ with ${\rm cov}(\widetilde{S})=\delta$, obtained by \cite[Theorem 5.80]{AS2}. Using that $S_0$ is a lifting of $T$, as well as the canonical representation of $\widetilde{S}$ on $\ka=\n(\Delta_{\widetilde{S}})\oplus \R(\Delta_{\widetilde{S}})$ given by \cite[Proposition 5.12]{AS2}, we obtain $\widetilde{S}$ in the form
$$
\widetilde{S}=
\begin{pmatrix}
S_0 & \star\\
0 & \star
\end{pmatrix}
\begin{bmatrix}
\ka_0 \\
\ka_1
\end{bmatrix}
=
\begin{pmatrix}
S_+ & \star & \star\\
0 & T & \star\\
0 & 0 & \star
\end{pmatrix}
\begin{bmatrix}
\h^{\perp}\\
\h\\
\ka_1
\end{bmatrix}
=
\begin{pmatrix}
\widetilde{V} & \delta \widetilde{E}\\
0 & U
\end{pmatrix}
\begin{bmatrix}
\n(\Delta_{\widetilde{S}})\\
\R(\Delta_{\widetilde{S}})
\end{bmatrix}.
$$
Here $S_+$ is from the matrix of $S_0$ in \eqref{eq23}, $\widetilde{V}$ and $\widetilde{E}$ are isometries with $\n(\widetilde{V}^*)=\R(\widetilde{E})$, $U$ is unitary and $\delta$ is as above. The subspace $\h^{\perp}$ is invariant for $S_0$ and so for $\widetilde{S}$ and $\widetilde{S}|_{\h^{\perp}}=S_0|_{\h^{\perp}}=S_+$. Hence $\h^{\perp} \subset \n(\Delta_{\widetilde{S}})$ (having in view that $\widetilde{S}$ is a 2-isometry) and $\widetilde{V}|_{\h^{\perp}}=S_+$. Then the operator $\widetilde{T}$ with $\widetilde{T}^*=\widetilde{S}^*|_{\h\oplus \ka_1}$ is an extension of $T$ on $\m=\h\oplus \ka_1$. Clearly, $\widetilde{T}$ has a block matrix on $\m=(\n(\Delta_{\widetilde{S}})\ominus \h^{\perp})\oplus \R(\Delta_{\widetilde{S}})=:\m_0\oplus \m_1$ of the form \eqref{eq24}, where $C=P_{\m_0}\widetilde{V}|_{\m_0}$, $E=P_{\m_0}\widetilde{E}|_{\m_1}$ and $U,\delta$ are as above.

For the condition \eqref{eq25} we decompose $\widetilde{V}$ on $\h^{\perp}\oplus \m_0$ and $\widetilde{E}:\R(\Delta_{\widetilde{S}})\to \h^{\perp} \oplus \m_0$ as
$$
\widetilde{V}=
\begin{pmatrix}
S_+ & C'\\
0 & C
\end{pmatrix},
\quad \widetilde{E}=
\begin{pmatrix}
E'\\
E
\end{pmatrix}.
$$
Here $C,C',E,E'$ are contractions ($C,E$ as above) with $S_+^*C'=0$, $C'^*C'=D_C^2$ and $E'^*E'=D_E^2$ (because $\widetilde{V}$ and $\widetilde{E}$ are isometries). In addition, as $\widetilde{V}^*\widetilde{E}=0$, it follows that $S_+^*E'=0$ and $C'^*E'+C^*E=0$. Thus, using the polar decomposition, we have $C'=J_CD_C$, $E'=J_ED_E$, where $J_C:\D_C\to \n(S_+^*)$ and $J_E:\D_E\to \n(S_+^*)$ are isometries with $\R(J_C)=\overline{\R(C')}$ and $\R(J_E)=\overline{\R(E')}$. We used here that $\R(C')\cup \R(E') \subset \n(S_+^*)$. Then the above condition becomes $D_CJ_C^*J_ED_E+C^*E=0$, i.e. the condition \eqref{eq25}. The proof is complete.
\end{proof}

Note that in \cite{McC} a Brownian extension
for concave operators $T$ with $\|T\|\le \sqrt{2}$ was obtained in a different way. In fact, the Brownian operators from \cite{McC} are concave, and (up to a normalization) they do model all concave operators. Concerning the operator $\widetilde{T}$ in \eqref{eq24}, it is easy to see that it is concave if and only if $C$ is an isometry with $C^*E=0$ and $U^*E^*EU\le E^*E$.

Next we turn to the class of concave operators $T$ which are $\Delta_T$-regular, that is satisfy the condition \eqref{eq13}. We have the following characterizations for these operators.

\begin{theorem}\label{te23}
For a non-isometric concave operator $T\in \B$ the following statements are equivalent:
\begin{itemize}
\item[(i)] $T$ is $\Delta_T$-regular;
\item[(ii)] $T$ has on the decomposition $\h=\n(\Delta_T)\oplus \overline{\R(\Delta_T)}$ the representation
\begin{equation}\label{eq26}
T=
\begin{pmatrix}
V & \sigma Z\\
0 & \widehat{T}
\end{pmatrix}
\end{equation}
where $V$ is an isometry, $Z$ is an injective contraction with $V^*Z=0$ and $\sigma^2=\|\Delta_T\|+1$, while $\widehat{T}$ is a contraction which commutes with $\sigma ^2 Z^*Z+\Delta_{\widehat{T}}$;
\item[(iii)] $T$ has a 2-isometric lifting $S$ on $\ka \supset \h$ with $S^*S \h \subset \h$ and $\n(\Delta_T)\subset \n(\Delta_S)$, such that $\sigma^{-2} \Delta_S$ is an orthogonal projection.
\end{itemize}
\end{theorem}

\begin{proof}
Assume that $T$ is $\Delta_T$-regular. Since $\delta:=\|\Delta_T\|^{1/2}>0$ and $T^*\Delta_T T\le \Delta_T$, the subspace $\n(\Delta_T)$ is invariant for $T$. Therefore $T$ has a block matrix of the form \eqref{eq26} with $V$ an isometry and two appropriate operators $Z, \widehat{T}$, and some scalar $\sigma>0$ which will be determined. As $\Delta_T\ge 0$, we obtain by \eqref{eq26} that $V^*Z=0$ and so $\Delta_T=0\oplus \Delta_0$, where $\Delta_0=\sigma^2Z^*Z+\Delta_{\widehat{T}}$ is an injective positive operator. It follows from the assertion (i) that $\Delta_0\widehat{T}=\Delta_0^{1/2}\widehat{T}\Delta_0^{1/2}$, which means (by the injectivity of $\Delta_0$) that $\Delta_0^{1/2}\widehat{T}=\widehat{T}\Delta_0^{1/2}$. Also, from $T^*\Delta_TT\le \Delta_T$ one obtains $\widehat{T}^*\Delta_0\widehat{T}\le \Delta_0$. This inequality together with the previous equality imply that we have
$$
\|\widehat{T}\Delta_0^{1/2}k\|=\|\Delta_0^{1/2}\widehat{T}k\|\le \|\Delta_0^{1/2}k\|
$$
for $k\in \overline{\R(\Delta_T)}$. Hence $\widehat{T}$ is a contraction. On the other hand, if $D_{\widehat{T}}$ is the defect operator of $\widehat{T}$, we have $\sigma^2Z^*Z=\Delta_0+D_{\widehat{T}}^2$. We infer $\sigma^2\|Z\|^2\le \|\Delta_0\|+1=\delta^2+1$, so choosing $\sigma=\sqrt{\delta^2+1}$ we get that $Z$ is a contraction. Also, one has $\n(Z)=\n(\Delta_0) \cap \n(D_{\widehat{T}})=\{0\}$. Therefore $Z$ is injective and the statement (ii) is proved. Hence (i) implies (ii).

Suppose now that $T$ has the form \eqref{eq26}. Let $\widehat{V}$ be the minimal isometric lifting of $\widehat{T}$ on $\ka_1=\overline{\R(\Delta_T)} \oplus l_+^2(\D_{\widehat{T}})=:\h_1\oplus \h_2$ (see \cite{FF, SFb}). Consider the operators
$$
Z_0:\h_1 \oplus \h_2 \to \n(\Delta_T)=:\h_0, \quad Z_0= \begin{pmatrix} \sigma Z & 0 \end{pmatrix}
$$
and
$$
Z_1:\h_1 \oplus \h_2 \to \ka_0:= l_+^2(\D_Z\oplus \h_2), \quad Z_1= \sigma \begin{pmatrix} J_1D_Z & J_2 \end{pmatrix},
$$
where $Z$ is the contraction from \eqref{eq26}, while $J_1: \D_Z\to \ka_0$ and $J_2: \h_2 \to \ka_0$ are the corresponding embedding mappings. Define the operator $S_1$ on $\ka=\ka_0 \oplus \h_0 \oplus \ka_1=(\ka_0 \oplus \h_0)\oplus \ka_1$ with the corresponding block matrices
$$
S_1=
\begin{pmatrix}
S_+^0 & 0 & Z_1\\
0 & V & Z_0\\
0 & 0 & \widehat{V}
\end{pmatrix}
=\begin{pmatrix}
V_1 & \sigma \widehat{Z}\\
0 & \widehat{V}
\end{pmatrix}.
$$
Here $S_+^0$ is the forward shift on $\ka_0$ and $V$ is as in \eqref{eq26}, while $V_1=S_+^0\oplus V$ on $\ka_0 \oplus \h_0$ and the operator $\widehat{Z}: \h_1\oplus \h_2\to \ka_0 \oplus \h_0$ is given by the column matrix $\widehat{Z}=\sigma^{-1} \begin{pmatrix} Z_1 & Z_0 \end{pmatrix} ^{\rm tr}$.

It is easy to see that $\widehat{Z}$ is an isometry and $V_1^*\widehat{Z}=0$ because $V^*Z=0$ in \eqref{eq26}. As $V_1$ and $\widehat{V}$ are isometries too, it follows that $S_1$ is a 2-isometry with
$\Delta_{S_1}=0\oplus \sigma ^2 I=\sigma ^2 P,$ where $P$ is the orthogonal projection onto $\ka_1 =\R(\Delta_{S_1})$. To see that $S_1$ is a lifting of $T$, we write explicitly $\widehat{V}$ in the above $3\times 3$ matrix of $S_1$. So, by using \eqref{eq26}, we obtain the representation
\begin{equation}\label{eq27}
S_1=
\begin{pmatrix}
W & X\\
0 & T
\end{pmatrix}
\begin{bmatrix}
\ka_0\oplus \h_2\\
\h_0 \oplus \h_1
\end{bmatrix},
\quad W=
\begin{pmatrix}
S_+^0 & \sigma J_2\\
0 & S_+^1
\end{pmatrix}
\begin{bmatrix}
\ka_0\\
\h_2
\end{bmatrix},
\quad X=
\begin{pmatrix}
0 & \sigma J_1D_Z\\
0 & \widetilde{D}_{\widehat{T}}
\end{pmatrix}
\begin{bmatrix}
\h_0\\
\h_1
\end{bmatrix},
\end{equation}
where $S_+^1$ is the forward shift on $\h_2$ and $\widetilde{D}_{\widehat{T}}=\widehat{J}D_{\widehat{T}}$ with $\widehat{J}$ the embedding mapping of $\D_{\widehat{T}}$ into $\h_2$.
This shows that $S_1$ is a lifting of $T$ with $W^*X=0$ because $J_2^*J_1D_Z=0$, $\R(J_1)$ and $\R(J_2)$ being orthogonal in $\ka_0$. Thus on $\ka=\ka_0 \oplus \h_2\oplus \h_0\oplus \h_1$ one obtains
$$
S_1^*S_1=W^*W\oplus (X^*X+\Delta_T)=I_{\ka_0}\oplus (\sigma^2+1)I_{\h_2} \oplus I_{\h_0}\oplus (\sigma ^2+1)I_{\h_1}.
$$
Hence $\n(\Delta_T)=\h_0 \subset \n(\Delta_{S_1})$ and $\overline{\R(\Delta_T)}=\h_1\subset \R(\Delta_{S_1})$. In conclusion, $S_1$ has the properties from (iii). Therefore (ii) implies (iii).

Finally, we assume that there exists a lifting $S_1$ for $T$ as in (iii). So $S_1$ has the form \eqref{eq27} for some 2-isometry $W$ on $\h^{\perp}=\ka \ominus \h$ and an operator $X:\h\to \h^{\perp}$ with $W^*X=0$. Therefore $\Delta_{S_1}=\Delta_W\oplus (X^*X+\Delta_T)=\sigma^2P$ and $P=P_{\R(\Delta_{S_1})}$. Since $(\sigma^{-2}\Delta_{S_1})^2=\sigma^{-2}\Delta_{S_1}$, we infer $\Delta_W^2=\sigma^2\Delta_W$. Therefore $\Delta_W=\sigma^2P_0$, where $P_0=P_{\R(\Delta_W)}$. We used here that $W$ is a $2$-isometry, $\Delta_W\ge 0$, and that $\n(\Delta_W)\subset \n(\Delta_{S_1})$.

On the other hand, we always have $\n(\Delta_{S_1}) \cap \h \subset \n(\Delta_T)$ because
$$
0\le \langle \Delta_T h,h \rangle=\langle S_1^*P_{\h}S_1h,h\rangle -\|h\|^2\le \langle \Delta_{S_1}h,h \rangle
$$
for $h\in \h$. Since we have also $\n(\Delta_T)\subset \n(\Delta_{S_1})$  from (iii), it follows that $\n(\Delta_T)=\n(\Delta_{S_1})\cap \h= \n(X^*X+\Delta_T)$. Therefore $\n(\Delta_T)$ is invariant for $S_1$. On the other hand, the above inequality $\Delta_T \le \Delta_{S_1}|_{\h}$ gives $\R(\Delta_T)\subset \Delta_{S_1}\h \subset \R(\Delta_{S_1})\cap \h$ which by the previous equality becomes $\overline{\R(\Delta_T)}=\R(\Delta_{S_1})\cap \h=\overline{\R(X^*X+\Delta_T)}$. Thus $\overline{\R(\Delta_T)}$ reduces $S_1^*S_1$ and one has $\Delta_{S_1}=\sigma^2P_0\oplus \sigma^2 P_T$, where $P_T\in \B$ is the orthogonal projection onto $\overline{\R(\Delta_T)}$. In addition, since $S_1$ is a 2-isometry, $T$ as a compression of $S_1$ to $\h$ is a $P_T$-contraction. As $\n(\Delta_T)$ is invariant for $T$ one also has the relation $P_TT=P_TTP_T$, i.e. $T$ is $P_T$-regular.

Now by the inequality $T^*P_TT\le P_T$ there exists a contraction $T_0$ on $\R(P_T)=\overline{\R(\Delta_T)}$ satisfying the relation $T_0P_T=P_TT$. So $T_0$ is even the compression of $T$ to $\overline{\R(\Delta_T)}$. Then the usual representation of the concave operator $T$ on $\h=\n(\Delta_T)\oplus \overline{\R(\Delta_T)}$ is
$$
T=
\begin{pmatrix}
V & F\\
0& T_0
\end{pmatrix},
$$
with $V$ an isometry and some operator $F$ satisfying the condition $V^*F=0$ (as $\Delta_T\ge 0$). By $P_T$-regularity of $T$ we have $T_0P_Th=P_TTh=P_TT_0P_Th$ whence $T_0P_Th=P_TT_0h$ for $h\in \overline{\R(\Delta_T)}$. Since $\sigma ^2 P_T={\Delta_{S_1}}|_{\h}=0\oplus \Delta_0$, where $\Delta_0=F^*F+\Delta_{T_0}={\Delta_T}|_{\overline{\R(\Delta_T)}}$, it follows that $T_0\Delta_0=\Delta_0T_0$. Finally, as $\Delta_T=0\oplus \Delta_0$, we infer that $\Delta_TT=\Delta_{T}^{1/2}T\Delta_{T}^{1/2}$, i.e. $T$ is $\Delta_T$-regular. Thus we proved that (iii) implies (i), and this ends the proof.
\end{proof}

As a consequence, we re-obtain the equivalence of (i) with (ii) for a 2-isometry given in \cite[Proposition 5.1]{MMS}; see also \cite[Theorem 7.1]{ACJS}.
Notice that in the terminology of \cite{ACJS} a $\Delta_T$-regular 2-isometry $T$ is called a {\it quasi-Brownian isometry}.

\begin{corollary}\label{co24}
A non-isometric operator $T\in \B$ is a $\Delta_T$-regular 2-isometry if and only if $T$ has on $\h=\n(\Delta_T)\oplus \overline{\R(\Delta_T)}$ a representation \eqref{eq26} with $V$ and $\widehat{T}$ isometries, $Z$ an injective contraction such that $V^*Z=0$, $\widehat{T}Z^*Z=Z^*Z\widehat{T}$ and $\sigma^2=\|\Delta_T\|$.
\end{corollary}

\begin{remark}\label{re25}
\rm
In order to compare the 2-isometric liftings obtained in Theorem \ref{te23}, (iii) and in Theorem \ref{te21}, (i), we record the following remarks. For the 2-isometric lifting $S_1$ obtained in (the proof of) Theorem \ref{te23}, (iii), the subspaces $\n(\Delta_T)$ and $\overline{\R(\Delta_T)}$ reduce $S_1^*S_1$ (not just $\h$), while the covariance of $S_1$ is less or equal to the covariance of the lifting from Theorem \ref{te21}, (i). It is easy to see that if the concave operator $T$ possess a 2-isometric lifting $S$ which (only) satisfies the property that $\beta^{-2}\Delta_S$ is an orthogonal projection for some scalar $\beta >0$, then $\R(\Delta_T) \subset \R(\Delta_S)$. Also, $\overline{\R(\Delta_T)}$ is invariant for $S^*$ and $T_*:=T^*|_{\overline{\R(\Delta_T)}}=S^*|_{\overline{\R(\Delta_T)}}$ is a contraction. But this does not ensure that $T_*$ commutes to $\Delta_0:={\Delta_T}|_{\overline{\R(\Delta_T)}}$, in general, so $T$ in not $\Delta_T$-regular in this case.
\end{remark}

On the other hand, for any contraction $T_0$ on $\h$ one can obtain a non-isometric concave lifting $T$ on $\ka=l_+^2(\h)\oplus \h$ which is $\Delta_T$-regular, so of the form \eqref{eq26} with $T_0$ instead of $\widehat{T}$, $V$ a shift operator, and $\sigma Z$ an isometry with $\sigma=1+\|T_0\|^2$. Therefore $\Delta_T$-regularity of $T$ does not force $\widehat{T}$ in \eqref{eq26} to belong to a restrictive class of contractions.

\medskip

\section{The Cauchy dual of a regular concave operator}
\label{Sect:3}
\medskip

By definition, the {\it Cauchy dual} of a left invertible operator $T\in\B$ is the operator $T'=T(T^*T)^{-1}$. Recall (see \cite{Sh}) that if $T$ is an left invertible operator on $\h$, then the operator $T^*T$ is invertible and $T'^*=(T^*T)^{-1}T^*$ is a left inverse of $T$, hence $T'$ is left invertible too. Also, $T'$ is a contraction if $T$ is expansive. It is known from \cite[Theorem 2.9]{Ch} that if $T$ is concave, then $T'$ is hyponormal, i.e. $T'T'^*\le T'^*T'$. This implies that $T'$ is a $D_{T'}^2$-contraction, that is a 2-hypercontraction. Indeed, one has $D_{T'}^2\le D_{T'^*}^2$. Therefore
$$T'^*D_{T'}^2T'\le T'^*D_{T'^*}^2T'=T'^*T'-(T'^*T')^2 \le D_{T'}^2.$$
In this case, according to the terminology from $A$-contractions (see \cite{S-2006, S-2009}) we say that $T'$ is $D_{T'}^2$-{\it regular} if it satisfies the condition $D_{T'}^2T'=D_{T'}T'D_{T'}$.
The next result shows that the mapping $T \to T'$ preserves the regularity condition.

\begin{theorem}\label{te31}
There is a bijective mapping between the set of all $\Delta_T$-regular concave operators $T$ on $\h$ and the set of all left invertible 2-hypercontractions $T'$ on $\h$ which are $D_{T'}^2$-regular, with $\|P_{\D_{T'}}T'h\|\le \|T'^*T'h\|$ for $h\in \D_{T'}$.
\end{theorem}

\begin{proof}
The required mapping is given by $\varphi(T) = T'$, where $T'$ is the Cauchy dual of $T$.
Assume firstly that $T$ is a $\Delta_T$-regular concave operator. So by Theorem \ref{te23} (ii), $T$ has the representation \eqref{eq26} with $V$ an isometry, $Z$ and $\widehat{T}$ contractions such that $V^*Z=0$ and $\widehat{T}\Delta=\Delta \widehat{T}$, where $\Delta:=\sigma ^2 Z^*Z+ \widehat{T}^*\widehat{T}$. Since $T^*T=I\oplus \Delta$ on $\n(\Delta_T) \oplus \overline{\R(\Delta_T)}$ and $T^*T$ is invertible it follows that $\Delta$ is invertible too, and so $\widehat{T}\Delta^{-1}=\Delta^{-1}\widehat{T}$.

Now using \eqref{eq26} we get that $T'=T(T^*T)^{-1}$ has on $\h=\n(\Delta_T)\oplus \overline{\R(\Delta_T)}$ the block representation
\begin{equation}\label{eq31}
T'=
\begin{pmatrix}
V & Z'\Delta^{-1}\\
0 & \widehat{T}\Delta^{-1}
\end{pmatrix},
\quad Z'=\sigma Z.
\end{equation}
But $T'$ is a contraction and the representation \eqref{eq31} gives $D_{T'}^2=I-T'^*T'=0\oplus \Delta'$, where
$$
\Delta'=I-\Delta^{-1}Z'^*Z'\Delta^{-1}-\Delta^{-1}\widehat{T}^*\widehat{T}\Delta^{-1}=I-\Delta^{-1}.
$$

Since $\Delta-I=\Delta_T|_{\overline{\R(\Delta_T)}}$ is injective, it follows that $\Delta'=\Delta^{-1}(\Delta-I)$ is injective. Therefore $\n(D_{T'})=\n(\Delta_T)$, so $\overline{\R(D_{T'})}=\overline{\R(\Delta_T)}$. Hence the above representations of $T'$ and $D_{T'}^2$ are given on the decomposition $\h=\n(D_{T'})\oplus \overline{\R(D_{T'})}$. As $T'$ is a $D_{T'}^2$-contraction, the $D_{T'}^2$-regularity of $T'$ will mean that $\widehat{T}\Delta^{-1}\Delta'=\Delta'\widehat{T}\Delta^{-1}$ (taking into account that $\Delta'$ is injective). But this last relation holds because $\widehat{T}$ commutes with both $\Delta^{-1}$ and $\Delta'$, while $\Delta^{-1}$ and $\Delta'$ obviously commute. Hence $T'$ is $D_{T'}^2$-regular. Now denoting $T_0'=\widehat{T}\Delta^{-1}$ we have, for $h\in \D_{T'}$,
$$
\|P_{\D_{T'}}T'h\|=\|T'_0h\|\le \|\Delta^{-1}h\|=\|(I-\Delta ')h\|=\|T'^*T'h\|.
$$
Therefore $T'$ satisfies all the properties from the statement of the theorem, and so the correspondence $T \to T'$ induces a well-defined mapping $\varphi$ between the corresponding sets.

To prove that this map $\varphi$ is onto, let $C$ be a left invertible 2-hypercontraction on $\h$ which is $D_C^2$-regular, with $C_0^*C_0\le D_0^2$, where $C_0:=P_{\D_C}C|_{\D_C}$ and $D_0:=C^*C|_{\D_C}$. Then $\n(D_C)$ is invariant for $C$ and $V:=C|_{\n(D_C)}$ is an isometry ($C$ being a contraction). Therefore $C$ has on the decomposition $\h=\n(D_C)\oplus \D_C$ the block matrix
$$
C=
\begin{pmatrix}
V & C_1\\
0 & C_0
\end{pmatrix}
$$
with $V^*C_1=0$ ($C$ and $V$ being as above). Then $C^*C=I\oplus (C_1^*C_1+C_0^*C_0)=I\oplus D_0$. As $C$ is left invertible, it follows that $C^*C$ is invertible, hence $D_0$ is an invertible contraction. Also, we have $D_C^2=0\oplus (I-D_0)$.

Let $T:=C'=C(C^*C)^{-1}$. Then $T^*T=(C^*C)^{-1}C^*C(C^*C)^{-1}=(C^*C)^{-1}$, whence $T'=T(T^*T)^{-1}=C$. To conclude the surjectivity of the map $\varphi$, we show that $T$ is a $\Delta_T$-regular concave. Firstly, we have $\Delta_T=(C^*C)^{-1}-I=(I-C^*C)(C^*C)^{-1}$. Hence $\n(\Delta_T)=\n(D_C)$ and $\R(\Delta_T)=\R(D_C)$. Thus, using the block matrix of $C$, we obtain the following representation of $T$ on $\h=\n(\Delta_T)\oplus \overline{\R(\Delta_T)}$:
$$
T=
\begin{pmatrix}
V & C_1D_0^{-1}\\
0 & C_0D_0^{-1}
\end{pmatrix}.
$$
Now, since $C$ is $D_C^2$-regular, i.e. $D_C^2C=D_CCD_C$, we infer that
$$
(I-D_0)C_0=(I-D_0)^{1/2}C_0(I-D_0)^{1/2}.
$$
Since $I-D_0={D_C^2}|_{\D_C}$ is injective, we deduce that $(I-D_0)C_0=C_0(I-D_0)$. This also gives $D_0C_0=C_0D_0$ and so $C_0D_0^{-1}=D_0^{-1}C_0$. But from the above block matrix of $T$ we get
$$
\Delta_T=T^*T-I=0\oplus [D_0^{-1}(C_1^*C_1+C_0^*C_0)D_0^{-1}-I]=0\oplus (D_0^{-1}-I).
$$
Then, denoting $\widehat{T}=C_0D_0^{-1}$, we have
$$
\widehat{T}(D_0^{-1}-I)=C_0(D_0^{-1}-I)D_0^{-1}=(D_0^{-1}-I)\widehat{T}.
$$
This relation implies that $\widehat{T}$ is a contraction and $T$ is concave. Indeed, since $C_0^*C_0\le D_0^2$, it follows that $D_0^{-1}C_0^*C_0D_0^{-1}\le I$, that is $\widehat{T}$ is a contraction. So, we get
\begin{eqnarray*}
T^*\Delta_TT&=& 0\oplus D_0^{-1}C_0^*(D_0^{-1}-I)C_0D_0^{-1}\\
&=& 0 \oplus (D_0^{-1}-I)^{1/2}D_0^{-1}C_0^*C_0D_0^{-1}(D_0^{-1}-I)^{1/2}\le 0 \oplus (D_0^{-1}-I)=\Delta_T,
\end{eqnarray*}
that is $T$ is concave. In addition, as $\widehat{T}$ commutes to $D_0^{-1}-I={\Delta_T}|_{\overline{\R(\Delta_T)}}$, we have by Theorem \ref{te23} (i) that $T$ is $\Delta_T$-regular. Thus $T$ has the required properties.

We conclude that the mapping $\varphi$ from the set of $\Delta_T$-regular concave operators $T$ into the set of $D_T^2$-regular 2-hypercontractions $C$ given by $\varphi(T)=T'(=C)$ is surjective, and it remains to show that $\varphi$ is injective. Indeed, let us assume that $T'=T'_1$ for two regular concave operators $T$ and $T_1$. Then $T=T'T^*T=T'_1T^*T$, which gives $T^*T=T^*T'_1T^*T$ so $I=T^*T'_1=T'^*_1T$. Using this and a previous relation, we obtain $I=T'^*_1T'_1T^*T$, whence
$$
(T^*T)^{-1}=T'^*_1T'_1=(T^*_1T_1)^{-1}T_1^*T_1(T^*_1T_1)^{-1}=(T^*_1T_1)^{-1},
$$
that is $T^*T=T^*_1T_1$. Finally, as $T'=T'_1$, we get
$
T=T'T^*T=T'_1T^*_1T_1=T_1.
$
Therefore $\varphi$ is injective. This ends the proof.
\end{proof}

\begin{remark}\label{re32}
\rm
Assume that $T$ is a concave operator. Then $T'$ is $D_{T'}^2$-isometry if and only if $T'$ is an isometry, and in this case $T=T'$. Therefore, even if $T$ is a $\Delta_T$-regular 2-isometry, we do not have more information about $T'$.

Recall that sometimes concave operators are called 2-hyperexpansive. Also, it is a well-known fact that an operator $T$ is $m$-hyperexpansive for $m\ge 2$ if and only if $T$ is $(\Delta_T,m-1)$-hypercontractive (see Section 1 for terminology). Now, assuming $\Delta_T$-regularity, we can express this equivalence in terms of the contraction $\widehat{T}$ from \eqref{eq26}.
\end{remark}

\begin{proposition}\label{pr33}
Let $T\in \B$ be a $\Delta_T$-regular concave operator and let $m\ge 2$ be an integer. Then $T$ is $m$-hyperexpansive if and only if the compression $\widehat{T}$ of $T$ on $\overline{\R(\Delta_T)}$ is a $(m-1)$-hypercontraction.
\end{proposition}

\begin{proof}
We use the representation \eqref{eq26} of $T$ on $\h=\n(\Delta_T) \oplus \overline{\R(\Delta_T)}$. Thus $\Delta_T=0\oplus \Delta_0$, where $\Delta_0=\sigma^2 Z^*Z+\Delta_{\widehat{T}}\ge 0$ is an injective operator, while $\widehat{T}=P_{\overline{\R(\Delta_T)}}T|_{\overline{\R(\Delta_T)}}$ is a contraction with $\widehat{T}\Delta_0=\Delta_0\widehat{T}$. One can easily prove by induction that the relation
$$
\sum_{j=0}^m(-1)^j\begin{pmatrix} m\\ j \end{pmatrix} T^{*j}T^j\le 0\quad {\rm i.e.} \hspace*{2mm} T \hspace*{2mm} {\rm is}\hspace*{2mm} m-{\rm expansive}
$$
is equivalent to
$$
\sum_{j=0}^{m-1} (-1)^j\begin{pmatrix} m-1\\ j \end{pmatrix} T^{*j}\Delta_TT^j\ge 0 \quad {\rm i.e.} \hspace*{2mm} T \hspace*{2mm} {\rm is} \hspace*{2mm} (\Delta_T,m-1)-{\rm contractive}.
$$
Then the last relation can be expressed in terms of $\widehat{T}$ and $\Delta_0$ as
$$
\sum_{j=0}^{m-1} (-1)^j\begin{pmatrix} m-1\\ j \end{pmatrix} \widehat{T}^{*j}\Delta_0\widehat{T}^j \ge 0.
$$
Using that $\widehat{T}\Delta_0=\Delta_0\widehat{T}$, this is equivalent to
$$
\Delta_0\left[\sum_{j=0}^{m-1} (-1)^j\begin{pmatrix} m-1\\ j \end{pmatrix} \widehat{T}^{*j}\widehat{T}^j\right]\Delta_0^{1/2} \ge 0.
$$
Since $\Delta_0$ is injective, this inequality is equivalent to
$$
\sum_{j=0}^{m-1} (-1)^j\begin{pmatrix} m-1\\ j \end{pmatrix} \widehat{T}^{*j}\widehat{T}^j \ge 0,
$$
which means that $\widehat{T}$ is $(m-1)$-contractive. This argument shows that for $m\ge 2$ and for $2\le n\le m-1$, $T$ is $n$-expansive if and only if $\widehat{T}$ is $(n-1)$-contractive. In other words, $T$ is $m$-hyperexpansive if and only if $\widehat{T}$ is $(m-1)$-hypercontractive. This ends the proof.
\end{proof}

Now, for $\Delta_T$-regular concave operators we obtain an affirmative answer to the Cauchy dual problem (see \cite[Question 2.11]{Ch}).

\begin{theorem}\label{te34}
Let $T\in \B$ be a $\Delta_T$-regular concave operator. The following statements are equivalent:
\begin{itemize}
\item[(i)] $T$ is completely hyperexpansive;
\item[(ii)] The Cauchy dual $T'$ of $T$ is subnormal;
\item[(iii)] The compression of $T$ on $\overline{\R(\Delta_T)}$ is subnormal;
\item[(iv)] The compression of $T'$ to $\D_{T'}$ is subnormal.
\end{itemize}
\end{theorem}

\begin{proof}
The assertion (i) means that $T$ is $m$-hyperexpansive for every integer $m\ge 2$ which by Proposition \ref{pr33} is equivalent to the fact that $\widehat{T}=P_{\overline{\R(\Delta_T)}}T|_{\overline{\R(\Delta_T)}}$ is an $m$-hypercontraction for any $m\ge 1$. According to a result of Agler from \cite{Ag}, the operator $\widehat{T}$ is subnormal. Hence (i) is equivalent to (iii).

Next we take into account the representation \eqref{eq31} of $T'$ on $\h=\n(D_{T'})\oplus \D_{T'}$, where $\n(D_{T'})=\n(\Delta_T)$ (see the proof of Theorem \ref{te31}). Denote $\widehat{T}'=\widehat{T}\Delta^{-1}=P_{\D_{T'}}T'|_{\D_{T'}}$, where $\Delta=\Delta_0+I$ and $\Delta_0=\Delta_T|_{\overline{\R(\Delta_T)}}$. Then $D_{T'}^2=0\oplus (I-\Delta^{-1})$ and it is easy to see that $T'$ is $m$-hypercontractive if and only if
$$
(I-\Delta^{-1})^{1/2}\left[\sum_{j=0}^{n} (-1)^j\begin{pmatrix} n\\ j \end{pmatrix} \widehat{T}'^{*j}\widehat{T}'^j\right](I-\Delta^{-1})^{1/2}\ge 0, \quad 1\le n\le m.
$$
We use here that $\widehat{T}'(I-\Delta^{-1})=(I-\Delta^{-1})\widehat{T}'$ because $\widehat{T}\Delta=\Delta\widehat{T}$ by Theorem \ref{te23} (ii). Since $(I-\Delta^{-1})^{1/2}=D_{T'}|_{\D_{T'}}$ is injective, the previous inequality is equivalent to the fact that $\widehat{T}'$ is $n$-contractive for $1\le n\le m$, that is to the fact that $\widehat{T}'$ is $m$-hypercontractive. So, by Agler's result (see \cite{Ag}), we infer that the statements (ii) and (iv) are equivalent.

Finally, if $\widehat{T}$ is subnormal then, because $\widehat{T}'=\widehat{T}\Delta^{-1}=\Delta^{-1}\widehat{T}$ and $\Delta^{-1}\ge 0$, it follows by Bram's result in \cite{B} that $\widehat{T}'$ is subnormal too. Conversely, if $\widehat{T}'$ is subnormal, then $\widehat{T}=\widehat{T}'\Delta=\Delta\widehat{T}'$ is subnormal by the same argument. Thus the statements (iii) and (iv) are equivalent. The proof is complete.
\end{proof}

In particular, for 2-isometries we re-obtain the main assertion of \cite[Theorem 7.5]{ACJS} which was proved there in a different manner. As we already mentioned, in \cite{ACJS} the regular 2-isometries are called quasi-Brownian isometries. Clearly such operators are completely hyperexpansive. Thus by Theorem \ref{te34} we have the following

\begin{corollary}\label{co35}
The Cauchy dual $T'$ of a $\Delta_T$-regular 2-isometry $T$ is a subnormal contraction.
\end{corollary}

As a direct consequence of \cite[Proposition 5.6]{AS2} and Theorem \ref{te34} we have the following result which generalizes \cite[Corollary 7.6]{ACJS}.

\begin{corollary}\label{co36}
If $T$ is a concave operator with $\Delta_T$ of rank one, then $T$ is a $\Delta_T$-regular completely hyperexpansive operator and $T'$ is a subnormal contraction.
\end{corollary}
\medskip

\section{Quasinormality conditions}
\medskip

We study now when the compressions of $T$ and $T'$ from the statements (iii) and (iv) of Theorem \ref{te34} are quasinormal.

\begin{theorem}\label{te41}
Let $T\in \B$ be a concave operator and let $m,n\ge 1$ be positive integers. Then $T^n$ is a $\Delta_{T^m}$-contraction and $\n(\Delta_T)=\n(\Delta_{T^n})$.

Moreover, if $T$ is $\Delta_T$-regular, then the following statements are equivalent:
\begin{itemize}
\item[(i)] $T^n$ is $\Delta_{T^m}$-regular for every $m,n\ge 1$;
\item[(ii)] The compression $\widehat{T}$ of $T$ to $\overline{\R(\Delta_T)}$ is quasinormal;
\item[(iii)] $\widehat{T}$ commutes with $\widehat{Z}^*\widehat{Z}$ where $\widehat{Z}=P_{\n(\Delta_T)}T|_{\overline{\R(\Delta_T)}}$;
\item[(iv)] The Cauchy dual $T'$ of $T$ is a regular $D_{T'^2}^2$-contraction;
\item[(v)] The compression $T'_0$ of $T'$ to $\overline{\R(\Delta_T)}$ is quasinormal;
\item[(vi)] $T'_0$ commutes with $T'^*_1T'_1$, where $T'_1=P_{\n(\Delta_T)}T'|_{\overline{\R(\Delta_T)}}$.
\end{itemize}
\end{theorem}

\begin{proof}
Assume that $T$ is concave, that is a $\Delta_T$-contraction. So, one has $T^{*2}T^2-T^*T\le T^*T-I$. Then $T^2$ is also a $\Delta_T$-contraction, therefore we have $T^{*3}T^3-T^{*2}T^2\le T^*T-I$. Both these relations and the fact that $T$ is expansive give that $T$ is a $\Delta_{T^2}$-contraction. Indeed, we have
$$
T^{*3}T^3-T^{*2}T^2=T^*(T^{*2}T^2-T^*T)T\le T^*(T^*T-I)T=T^{*2}T^2-T^*T.
$$
We obtain
\begin{eqnarray*}
T^*\Delta_{T^2}T &=& T^{*3}T^3-T^*T=T^{*3}T^3-T^{*2}T^2+T^{*2}T^2-T^*T\\
& \le & T^{*2}T^2-T^*T+T^*T-I=\Delta_{T^2}.
\end{eqnarray*}
Since $\Delta_{T^2}\ge 0$ ($T$ being expansive) we conclude that $T$ is a $\Delta_{T^2}$-contraction.

We can now show by induction that $T$ is a $\Delta_{T^m}$-contraction for each integer $m\ge 2$. So, assuming that $T^*\Delta_{T^m}T \le \Delta_{T^m}$ for $m>2$, we have
\begin{eqnarray*}
T^*\Delta_{T^{m+1}}T&=& T^{*(m+2)}T^{m+2}-T^*T=T^{*2}\Delta_{T^m}T^2+T^{*2}T^2-T^*T\\
&\le & T^*\Delta_{T^m}T+\Delta_T=T^{*(m+1)}T^{m+1}-I=\Delta_{T^{m+1}}.
\end{eqnarray*}
Hence $T$ is a $\Delta_{T^m}$-contraction for any $m\ge 1$ and consequently $T^n$ is a $\Delta_{T^m}$-contraction for $m,n\ge 1$. In addition,
 because $T$ is expansive, one has $\Delta_T\le \Delta_{T^m}$. Therefore $\n(\Delta_{T^m})\subset \n(\Delta_T)$. But $\n(\Delta_T)$ is invariant for $T$, so also for $T^m$ and $T^m|_{\n(\Delta_T)}$ is an isometry, $T|_{\n(\Delta_T)}$ being so. Since $\n(\Delta_T)$ is also invariant for $T^{*m}T^m$ it follows that $\n(\Delta_T)\subset \n(\Delta_{T^m})$. We conclude that $\n(\Delta_{T^m})=\n(\Delta_T)$. The first assertion of theorem is proved.

Assume now that $T$ is $\Delta_T$-regular and that the statement (i) is true. Then $T$ is also $\Delta_{T^2}$-regular, as a $\Delta_{T^2}$-contraction, while $T$ and $T^2$ have on $\h=\n(\Delta_T)\oplus \overline{\R(\Delta_T)}=\n(\Delta_{T^2})\oplus \overline{\R(\Delta_{T^2})}$ the block matrices
$$
T=
\begin{pmatrix}
V & \widehat{Z}\\
0 & \widehat{T}
\end{pmatrix},
\quad T^2=
\begin{pmatrix}
V^2 & V\widehat{Z}+\widehat{Z}\widehat{T}\\
0 & \widehat{T}^2
\end{pmatrix}
$$
with $V^*\widehat{Z}=0$. Consequently, $V^{*2}(V\widehat{Z}+\widehat{Z}\widehat{T})=0$. But, by Theorem \ref{te23} (ii), one has $\widehat{T}(\widehat{Z}^*\widehat{Z}+\widehat{T}^*\widehat{T})=(\widehat{Z}^*\widehat{Z}+\widehat{T}^*\widehat{T})\widehat{T}$. Also, since $T^{*2}T^2=I\oplus [\widehat{Z}^*\widehat{Z}+\widehat{T}^*(\widehat{Z}^*\widehat{Z}+\widehat{T}^*\widehat{T})\widehat{T}]$ and $T$ is $\Delta_{T^2}$-regular, i.e. $\Delta_{T^2}T=\Delta_{T^2}^{1/2}T\Delta_{T^2}^{1/2}$, we infer that $\widehat{T}$ is a $\Delta_1$-contraction and $\Delta_1$-regular, where $\Delta_1:=\Delta_{T^2}|_{\overline{\R(\Delta_{T^2})}}=\Delta_{T^2}|_{\overline{\R(\Delta_T)}}$. Having in view the above expression of $T^{*2}T^2$ and the fact that $\Delta_1$ is a positive injective operator, we infer that
$$
\widehat{T}[\widehat{Z}^*\widehat{Z}+\widehat{T}^*(\widehat{Z}^*\widehat{Z}+\widehat{T}^*\widehat{T})\widehat{T}]=[\widehat{Z}^*\widehat{Z}+\widehat{T}^*
(\widehat{Z}^*\widehat{Z}+\widehat{T}^*\widehat{T})\widehat{T}]\widehat{T}.
$$
Using that $\widehat{T}\Delta=\Delta\widehat{T}$, where $\Delta:=T^*T|_{\overline{\R(\Delta_T)}}=\widehat{Z}^*\widehat{Z}+\widehat{T}^*\widehat{T}$, we can equivalently write the previous relation in the form
$$
\widehat{T}\widehat{Z}^*\widehat{Z}-\widehat{Z}^*\widehat{Z}\widehat{T} + (\widehat{T} \widehat{T}^*\widehat{T}- \widehat{T}^*\widehat{T}^2)\Delta=0.
$$
By the commutation $\widehat{T}\Delta=\Delta\widehat{T}$ we have $\widehat{T}\widehat{Z}^*\widehat{Z}-\widehat{Z}^*\widehat{Z}\widehat{T}= \widehat{T}^*\widehat{T}^2-\widehat{T} \widehat{T}^*\widehat{T}$, which together with the above relation lead to
$$
(\widehat{T} \widehat{T}^*\widehat{T}- \widehat{T}^*\widehat{T}^2)(\Delta-I)=0.
$$
Because $\Delta-I=\Delta_T|_{\overline{\R(\Delta_T)}}$ is positive and injective, one has $\overline{\R(\Delta-I)}=\overline{\R(\Delta_T)}$. So from the previous equality we infer that the contraction $\widehat{T}$ is quasinormal. Hence (i) implies (ii).

Next, we assume (ii), that is, $\widehat{T}$ is quasinormal. Since $T$ is $\Delta_T$-regular we have $\widehat{T}(\widehat{Z}^*\widehat{Z}+\widehat{T}^*\widehat{T})=(\widehat{Z}^*\widehat{Z}+\widehat{T}^*\widehat{T})\widehat{T}$. By our assumption we get that $\widehat{T}\widehat{Z}^*\widehat{Z}=\widehat{Z}^*\widehat{Z}\widehat{T}$. Therefore (ii) implies (iii). Also, the last commutation relation of $\widehat{T}$ to $\widehat{Z}^*\widehat{Z}$ implies, using the $\Delta_T$-regularity of $T$, that $\widehat{T}$ is quasinormal. So the assertions (ii) and (iii) are equivalent.

To prove that (ii) implies (i) we represent $T^m$ on $\h=\n(\Delta_T)\oplus \overline{\R(\Delta_T)}$ using the above block matrix of $T$, in the form
$$
T^m=
\begin{pmatrix}
V^m & \sum_{j=0}^{m-1}V^{m-j-1}\widehat{Z}\widehat{T}^j\\
0 & \widehat{T}^m
\end{pmatrix}
=
\begin{pmatrix}
V^m & Z_m\\
0 & \widehat{T}^m
\end{pmatrix}.
$$
Since $V^*\widehat{Z}=0$, we have $V^{*m}Z_m=0$. From this representation we get, using $V^*\widehat{Z}=0$,
\begin{eqnarray*}
T^{*m}T^m|_{\overline{\R(\Delta_T)}}&=& Z_m^*Z_m+\widehat{T}^{*m}\widehat{T}^m=\sum_{j=0}^{m-1}\widehat{T}^{*j}\widehat{Z}^*\widehat{Z}T^j+\widehat{T}^{*m}\widehat{T}^m\\
&=& \Delta + \sum_{j=1}^{m-1}\widehat{T}^{*j}(\Delta-I)\widehat{T}^j=:\Delta_m.
\end{eqnarray*}
Here, as above, $\Delta=\widehat{Z}^*\widehat{Z}+\widehat{T}^*\widehat{T}$.

Now, as $T$ is $\Delta_T$-regular, we have $\widehat{T}^n\Delta=\Delta\widehat{T}^n$ for $n\ge 1$. Also, assuming that the assertion (ii) is true, i.e. $\widehat{T}$ is quasinormal, we infer that $\widehat{T}^n\Delta_m=\Delta_m \widehat{T}^n$ for $m,n\ge 1$. Recall that  $\widehat{T}^n=T^n|_{\overline{\R(\Delta_m)}}$ is a contraction ($\widehat{T}$ being so) and $\Delta_m-I=\Delta_{T^m}|_{\overline{\R(\Delta_m)}}$ because $\overline{\R(\Delta_T)}=\overline{\R(\Delta_{T^m})}$. We also have that $T^n$ is a $\Delta_{T^m}$-contraction and from the above commutation of $\widehat{T}^n$ with $\Delta_m$ it follows that $T^n$ is $\Delta_{T^m}$-regular, for $m,n\ge 1$. Hence (ii) implies (i).

Next we use the representation \eqref{eq31} of $T'$ and denote $T'_0=\widehat{T}\Delta^{-1}=P_{\overline{\R(\Delta_T)}}T'|_{\overline{\R(\Delta_T)}}$, where $\Delta$ is as above. We have that $T'_0$ is quasinormal if and only if $\widehat{T}$ is quasinormal, taking into account that $\widehat{T}\Delta^{-1}=\Delta^{-1}\widehat{T}$ by the $\Delta_T$-regularity of $T$. So (ii) is equivalent to (v), and similarly (iii) is equivalent to (vi), where $T'_1=Z'\Delta^{-1}=\widehat{Z}\Delta^{-1}$ in \eqref{eq31}. To end the proof we will prove that (ii) is equivalent to (iv).

Notice firstly that since $T'$ is a $D_{T'}^2$-contraction, $T'^2$ will be also a $D_{T'}^2$-contraction, i.e. $T'^{*2}D_{T'}^2T'^2\le D_{T'}^2$, because $T'^{*2}D_{T'}^2T'^2\le T'^*D_{T'}^2T'$. Moreover, we get that
$$
T'^*D_{T'^2}^2T'= T'^*D_{T'}^2T'+ T'^{*2}D_{T'}^2T'^2
 \le  D_{T'}^2+T'^*D_{T'}^2T'=D_{T'^2}^2.
$$
Hence $T'$ is a $D_{T'^2}^2$-contraction. Since $T'$ is a contraction, we have $D_{T'}^2\le D_{T'^2}^2$. Therefore $\n(D_{T'^2}^2)\subset \n(D_{T'})=\n(\Delta_T)$. But the last kernel is an invariant subspace for $T'$, so also for $T'^2$. From the block matrix \eqref{eq31} of $T'$ we have that $T'^2|_{\n(\Delta_T)}$ is an isometry. Thus we obtain that $\n(\Delta_T)=\n(D_{T'}^2)=\n(D_{T'^2}^2)$. Consequently one has $\overline{\R(\Delta_T)}=\D_{T'}=\D_{T'^2}$.

Using the block matrix \eqref{eq31} of $T'$ on $\h=\n(\Delta_T)\oplus \overline{\R(\Delta_T)}$ we get $T'^2$ in the form
$$
T'^2=
\begin{pmatrix}
V^2& V\widehat{Z}\Delta^{-1}+\widehat{Z}\Delta^{-1}\widehat{T}\Delta^{-1}\\
0 & \widehat{T}\Delta^{-1}\widehat{T}\Delta^{-1}
\end{pmatrix}
=
\begin{pmatrix}
V^2 & (V\widehat{Z}+\widehat{Z}\Delta^{-1}\widehat{T})\Delta^{-1}\\
0 & \widehat{T}^2\Delta^{-2}
\end{pmatrix}.
$$
We used here that $\widehat{T}\Delta^{-1}=\Delta^{-1}\widehat{T}$ by $\Delta_T$-regularity of $T$. This representation of $T'^2$ gives immediately that
\begin{eqnarray*}
T'^{*2}T'^2 &= & I\oplus \Delta^{-1}[\widehat{Z}^*\widehat{Z}+ \widehat{T}^*\Delta^{-1}(\widehat{Z}^*\widehat{Z}+\widehat{T}^*\widehat{T})\Delta^{-1}\widehat{T}]\Delta^{-1}\\
&=& I\oplus \Delta^{-1}(\widehat{Z}^*\widehat{Z}+\widehat{T}^*\Delta^{-1}\widehat{T})\Delta^{-1}\\
&=& I\oplus \Delta^{-1}(\widehat{Z}^*\widehat{Z}+\widehat{T}^*(I-\Delta')\widehat{T})\Delta^{-1}\\
&=& I\oplus \Delta^{-1}(\Delta-\widehat{T}^*\Delta'\widehat{T})\Delta^{-1}\\
&=& I\oplus (I-\widehat{T}^*\widehat{T}\Delta^{-1}\Delta')\Delta^{-1}.
\end{eqnarray*}
Here $\Delta':=I-\Delta^{-1}$ and we used also the fact that $\widehat{T}$ commutes with both $\Delta^{-1}$ and $\Delta'$. We infer that $D_{T'^2}^2=0\oplus (I+\widehat{T}^*\widehat{T}\Delta'\Delta^{-1})\Delta^{-1}$
and thus $D_{T'^2}^2$-regularity of $T'$ will mean that $T'_0=P|_{\overline{\R(\Delta_T)}}T'|_{\overline{\R(\Delta_T)}}=\widehat{T}\Delta^{-1}$ commutes with $D_{T'^2}^2|_{\overline{\R(\Delta_T)}}=(I+\widehat{T}^*\widehat{T}\Delta'\Delta^{-1})\Delta^{-1}$. Equivalently, this relation can be written as $\widehat{T}\widehat{T}^*\widehat{T}\Delta'\Delta^{-3}=\widehat{T}^*\widehat{T}^2\Delta'\Delta^{-3}$, which holds if and only if $\widehat{T}$ is quasinormal, since $\Delta'=I-\Delta^{-1}=(\Delta-I)\Delta^{-1}$ is injective by the injectivity of $\Delta-I=\Delta_T|_{\overline{\R(\Delta_T)}}$. This argument shows that $D_{T'^2}^2$-regularity of $T'$ is equivalent to the quasinormality of $\widehat{T}$. We conclude that the assertions (ii) and (iv) of theorem are equivalent. The proof is complete.
\end{proof}

\begin{remark}\label{re42}
\rm
It is easily seen that if $T$ is a $\Delta_T$-regular concave operator, then $T'$ (and consequently $T$) is quasinormal if and only  if $T$ is an isometry (and in this case $T=T'$).

On the other hand, by the first assertion of Theorem \ref{te41} a concave operator $T$ is a $\Delta_{T^2}$-contraction, while the assertion (i) of Theorem \ref{te41} ensures that $T$ is $\Delta_{T^2}$-regular if $T$ is $\Delta_T$-regular. But $\Delta_{T^2}$-regularity of $T$ is equivalent to each of the other statements (i)-(vi), because it implies (ii), as we have seen in the previous proof. Thus the general statement (i) can be reduced only to (i) for $n=1$ and $m=2$, what can be easily verified in applications.
\end{remark}

As a consequence of Theorem \ref{te41} and Theorem \ref{te34} we have the following result.

\begin{corollary}\label{co43}
Let $T$ be a $\Delta_T$-regular concave operator. If $T$ is $\Delta_{T^2}$-regular, then $T$ is completely hyperexpansive. In particular, if $\Delta_T$ is of rank one, then $T$ is $\Delta_{T^2}$-regular.
\end{corollary}

\begin{corollary}\label{co44}
\hspace*{7cm}
\begin{itemize}
\item[(a)] If $T$ is a regular concave operator that satisfies one of (the equivalent) conditions of Theorem \ref{te41}, then $T^n$ is a $\Delta_{T^n}$-regular concave operator for every integer $n\ge 2$.
\item[(b)] If $T$ is concave and $\sigma^{-1}P_{\n(\Delta_T)}T|_{\overline{\R(\Delta_T)}}$ is an isometry, where $\sigma^2=\|\Delta_T\|+1$, then $T$ is $\Delta_T$-regular if and only if $\widehat{T}= P_{\overline{\R(\Delta_T)}}T|_{\overline{\R(\Delta_T)}}$ is a quasinormal contraction.
\end{itemize}
\end{corollary}

\begin{proof}
The assertion (a) follows directly from the statement (i) of Theorem \ref{te41}.

For the assertion (b) we assume that $Z=\sigma^{-1}P_{\n(\Delta_T)}T|_{\overline{\R(\Delta_T)}}$ is an isometry. Then, by the matrix representation of the concave operator $T$ on $\h=\n(\Delta_T)\oplus \overline{\R(\Delta_T)}$ with $V=T|_{\n(\Delta_T)}$, $\widehat{Z}=\sigma Z$ and $\widehat{T}=P_{\overline{\R(\Delta_T)}}T|_{\overline{\R(\Delta_T)}}$, we have $\Delta_T|_{\overline{\R(\Delta_T)}}=(\sigma^2+1)I+\widehat{T}^*\widehat{T}$. Thus $\widehat{T}$ commutes to $\Delta_T|_{\overline{\R(\Delta_T)}}$ (i.e. $T$ is $\Delta_T$-regular) if and only if $\widehat{T}$ is quasinormal.
\end{proof}

The second statement of this corollary shows that if $Z$ before is an isometry, then the condition of $\Delta_T$-regularity can be added to the equivalent statements of Theorem \ref{te41}. If $T$ is a $\Delta_T$-regular 2-isometry (i.e. a quasi-Brownian isometry), then $\widehat{T}$ is even an isometry. Therefore $T$ trivially satisfies the assertions of Theorem \ref{te41} in this case.

Returning to Theorem \ref{te34}, we remark that for $T\in \B$ the assertion (i) of Theorem \ref{te34} means that $T$ is an $A_n(T)$-contraction for any integer $n\ge 1$, where as in \eqref{eq12},
$$
A_n(T)=-B_n(T)\ge 0, \quad B_n(T)=B_I^n(T), \quad A_1(T)=-B_1(T)=\Delta_T.
$$
Under the assumption of $\Delta_T$-regularity of $T$, this fact ensures that $\widehat{T}$ in \eqref{eq26} is subnormal, while $\Delta_{T^n}$-regularity of $T$ for any $n\ge 1$ means that $\widehat{T}$ is quasinormal, by Theorem \ref{te41}.

We analyse now the case when $T$ is $A_n(T)$-regular, that is $A_n(T)T=A_n(T)^{1/2}TA_n(T)^{1/2}$.
\begin{theorem}\label{te46}
Let $T\in \B$ be a $\Delta_T$-regular completely hyperexpansive operator.
With the above notation, the following statements are equivalent:
\begin{itemize}
\item[(i)] $T$ is $A_n(T)$-regular for every integer $n\ge 2$;
\item[(ii)] $T$ is $A_j(T)$-regular for $j=2,3$;
\item[(iii)] $T$ is $A_2(T)$-regular and the compression of $T$ to $\overline{A_2(T)\h}$ is a quasinormal contraction.
\end{itemize}

In addition, if (one of) the assertions (i)-(iii) hold, then $\n(A_2(T))=\n(A_n(T))$ for $n\ge3$.
\end{theorem}

\begin{proof}
By hypothesis and the above remark, we have that $T$ is an $A_n$-contraction, where $A_n:=A_n(T)$, for $n\ge 1$. This ensures that $\n(A_n)$ is an invariant subspace for $T$.

Obviously, (i) implies (ii). Assume now that (ii) holds, that is $T$ is $A_j$-regular for $j=2,3$, where
$$
A_2=\Delta_T-T^*\Delta_TT=I-2T^*T+T^{*2}T^2\ge 0,
$$
$$
A_3=A_2-T^*A_2T=I-3T^*T+3T^{*2}T^2-T^{*3}T^3\ge 0.
$$
Since $T$ is $\Delta_T$-regular, it follows from \cite[Proposition 2.1 and Theorem 4.6]{S-2006} that
$$
\n(A_2)=\n(\Delta_T)\oplus \n(I-S_{\widehat{T}})
$$
and that $\n(I-S_{\widehat{T}})$ is invariant for $\Delta_T$. Here $S_{\widehat{T}}$ is the asymptotic limit of the contraction $\widehat{T}=P_{\overline{\R(\Delta_T)}}T|_{\overline{\R(\Delta_T)}}$. But $\n(I-S_{\widehat{T}})$ is the maximum invariant subspace for $\widehat{T}$ on which $\widehat{T}$ is an isometry (see \cite{Ku}). In fact $\n(I-S_{\widehat{T}})=\n(I-\widehat{T}^*\widehat{T})$ because this last subspace is also invariant for $\widehat{T}$, $\widehat{T}$ being a $D_{\widehat{T}}^2$-contraction.

Using the block matrix \eqref{eq26} of $T$ and denoting $\Delta_0=\Delta_T|_{\overline{\R(\Delta_T)}}$, we have
$$
A_2=0\oplus (\Delta_0-\widehat{T}^*\Delta_0\widehat{T})=0\oplus (I-\widehat{T}^*\widehat{T})\Delta_0=0\oplus \Delta_0(I-\widehat{T}^*\widehat{T}),
$$
taking into account that $\widehat{T}\Delta_0=\Delta_0\widehat{T}$ (by Theorem \ref{te23}). Next we represent $\widehat{T}$ and $\Delta_0$ on $\overline{\R(\Delta_T)}=\n(D_{\widehat{T}})\oplus \D_{\widehat{T}}$ in the form
\begin{equation}\label{eq41}
\widehat{T}=
\begin{pmatrix}
\widehat{V} & D\\
0 & C
\end{pmatrix}, \quad
\Delta_0=\Delta_2\oplus \Delta_1,
\end{equation}
where $\widehat{V}$ is an isometry, $C, D$ are contractions, $\widehat{V}^*D=0$ and $\Delta_j\ge 0$ ($j=1,2$). Clearly, $\Delta_0$ can be written in this form because $\n(D_{\widehat{T}})$ is invariant for $\Delta_T=0\oplus \Delta_0$ on $\h=\n(\Delta_T)\oplus \overline{\R(\Delta_T)}$, so $\n(D_{\widehat{T}})$ reduces $\Delta_0$. Then we obtain from the relation $\widehat{T}\Delta_0=\Delta_0\widehat{T}$ that $C\Delta_1=\Delta_1C$. By the above expression of $A_2$ we get that
$$
A_2=0\oplus (I-D^*D-C^*C)\Delta_1=0\oplus \Delta_1(I-D^*D-C^*C)
$$
on $\h=\n(A_2)\oplus \D_{\widehat{T}}$, with $\D_{\widehat{T}}=\overline{\R(A_2)}$, because $\n(D_{\widehat{T}})=\n(A_2)\ominus \n(\Delta_T)$.

Next we use the expression of $A_3$ in terms of $A_2$ and the fact that $\n(A_3)$ is invariant for $T$. Thus, as $T$ is $A_2$-regular by (ii), we have by \cite[Proposition 2.1 and Theorem 4.6]{S-2006} that
$$
\n(A_3)=\n(A_2)\oplus \n(I-S_C),
$$
$C$ being the compression of $T$ to $\D_{\widehat{T}}=\overline{\R(A_2)}$. By \eqref{eq41} and the condition $\widehat{V}^*D=0$, we get that $\n(I-C^*C)\subset \n(I-\widehat{T}^*\widehat{T})$. Hence
$$
\n(I-S_C)=\n(I-C^*C)=\{0\}.
$$
We obtain $\n(A_3)=\n(A_2)$ and so $\overline{\R(A_3)}=\overline{\R(A_2)}$. Also, since $T$ is $A_2$-regular, we have
\begin{equation}\label{eq42}
C(I-D^*D-C^*)\Delta_1=(I-D^*D-C^*C)\Delta_1C.
\end{equation}

On the other hand, as $A_3=A_2-T^*A_2T$, we infer
$$
A_3=0\oplus [(I-D^*D-C^*C)\Delta_1-C^*(I-D^*D-C^*C)\Delta_1C]=:0\oplus \Delta_3
$$
on $\h=\n(A_2) \oplus \D_{\widehat{T}}=\n(A_3)\oplus \D_{\widehat{T}}$. As $T$ is $A_3$-regular by (ii) and $C=P_{\overline{\R(A_3)}}T|_{\overline{\R(A_3)}}$, we have as in the proof of Theorem \ref{te23} that $C\Delta_3=\Delta_3C$. This relation, together with \eqref{eq42} and with the fact that $C\Delta_1=\Delta_1C$, lead to the identity
$$
CC^*C(I-D^*D-C^*C)\Delta_1=C^*C^2(I-D^*D-C^*C)\Delta_1.
$$
This means that $CC^*CA_2|_{\overline{\R(A_2)}}=C^*C^2A_2|_{\overline{\R(A_2)}}$. Since $A_2|_{\overline{\R(A_2)}}$ is an injective positive operator, we conclude that $CC^*C=C^*C^2$, i.e. $C$ is quasinormal. Hence (ii) implies (iii).

Next we assume that (iii) holds, that is $T$ is $A_2$-regular and that $C$ in \eqref{eq41} is quasinormal. We show that $\n(A_2)=\n(A_n)$ and that $T$ is $A_n$-regular for $n\ge 3$.
Recall that $A_n=A_{n-1}-T^*A_{n-1}T$ and that $T$ is an $A_n$-contraction, so $\n(A_n)$ is invariant for $T$, for any $n\ge 2$.
Since $T$ is also $A_2$-regular by (iii), it follows that $C$ commutes with $A_2|_{\overline{\R(A_2)}}=(I-D^*D-C^*C)\Delta_1$, $D$ and $\Delta_1$ from \eqref{eq41}. This gives that $CA_2^{1/2}h=A_2^{1/2}Th$ for $h\in \h$. Since $T$ is an $A_2$-contraction, $C$ is the (unique) contraction on $\overline{\R(A_2)}$ induced by $T$. Then, using \cite[Proposition 2.1 and Theorem 4.1]{S-2006}, we have
$$
\n(A_3)=\n(A_2-T^*A_2T)=(A_2^{1/2})^{-1}\n(I-S_C).
$$
As we have seen before we have $\n(I-S_C)=\{0\}$. Therefore we get
$$
\n(A_3)=\{h\in \h:A_2^{1/2}h=0\}=\n(A_2), \quad \overline{\R(A_3)}=\overline{\R(A_2)}=\D_{\widehat{T}}.
$$
Recall that $\Delta_3=A_3|_{\overline{\R(A_3)}}$. Therefore
$
\Delta_3=A_2|_{\D_{\widehat{T}}}-C^*(A_2|_{\D_{\widehat{T}}})C.
$
Taking into account that $C$ is quasinormal and that it commutes with $A_2|_{\D_{\widehat{T}}}$, it follows that $C\Delta_3=\Delta_3C$. This means that $T$ is $A_3$-regular.

We show by induction that $\n(A_2)=\n(A_n)$ and that $T$ is $A_n$-regular for $n\ge 3$. We had proved this fact for $n=3$ and we assume now that $T$ is $A_j$-regular with $\n(A_j)=\n(A_2)$ for $3\le j\le n$. A simple computation shows that
$$
\Delta_m:=A_m|_{\overline{R(A_m)}}=\sum_{j=0}^{m-2}(-1)^j \begin{pmatrix} m-2\\ j \end{pmatrix} C^{*j}(A_2|_{\D_{\widehat{T}}})C^j
$$
for $m\ge 3$. So, as $C$ is quasinormal and it commutes to $A_2|_{\D_{\widehat{T}}}$, it follows that $C\Delta_m=\Delta_mC$ for $m\ge 3$. Furthermore, since for $m\le n$ one has $\overline{\R(A_m)}=\overline{\R(A_2)}=\D_{\widehat{T}}$, we have by this commutation relation that $C$ is the (unique) contraction on $\overline{\R(A_n)}$ induced by $T$ as an $A_n$-contraction. Then, by \cite[Proposition 2.1 and Theorem 4.1]{S-2006}, we get (as above)
$$
\n(A_{n+1})=\n(A_n-T^*A_nT)=(A_n^{1/2})^{-1}\n(I-S_C)=\n(A_n).
$$
Hence $C$ is the compression of $T$ to $\overline{\R(A_{n+1})}= \overline{\R(A_n)}$. Since $C\Delta_{n+1}=\Delta_{n+1}C$, we infer that $T$ is $A_{n+1}$-regular. We conclude that $T$ is $A_m$-regular and that $\n(A_m) =\n(A_2)$ for every integer $m\ge 2$. Thus (iii) implies (i) and all assertions are proved.
\end{proof}
From the above proof we can extract more information about the operator $C$.
\begin{corollary}\label{co47}
If $T$ is a completely hyperexpansive operator which is $A_j$-regular for $j=1,2,3$, then the compression of $T$ to $\overline{\R(A_2)}$ is a completely non isometric quasinormal contraction.
\end{corollary}

Note that if $T$ is a $\Delta_T$-regular 2-isometry,
then $A_n=0$ for $n\ge 2$. So Theorem \ref{te46} and Corollary \ref{co47} are meaningful only for completely hyperexpansive operators which are not $\Delta_T$-regular 2-isometries.
\subsection*{Acknowledgment}
The authors gratefully thank the Referee for the constructive
comments and recommendations
which definitely helped to improve the readability and quality of the paper.
The first named author was supported in part by the project FRONT of the French
National Research Agency (grant ANR-17-CE40-0021) and by the Labex CEMPI (ANR-11-LABX-0007-01).
The second named author was supported by a Project financed from Lucian Blaga University of Sibiu research
grants LBUS-IRG-2017-03.


\bibliographystyle{amsalpha}

\end{document}